\documentclass[12pt]{article}
\usepackage{amsmath,amsthm,amssymb,euscript,verbatim}
\newtheorem{theorem}{Theorem}

\newtheorem{corollary}{Corollary}

\usepackage{makeidx}         
\usepackage{graphicx}        
\usepackage{multicol}        
\usepackage[bottom]{footmisc}

\makeindex             
\usepackage{latexsym}
\usepackage{epic,eepic,pstricks}

\begin{document}
\title
{\bf  Tight Spherical Embeddings\\ 
(Updated Version)}
\author
{Thomas E. Cecil\thanks{Partially supported by a Summer Faculty Fellowship from the College of the Holy Cross}
\ and Patrick J. Ryan\thanks{Partially supported by a Summer Faculty Fellowship from Indiana University}}
\maketitle

\begin{abstract}
This is an updated version of the paper \cite{CR6} which appeared in the proceedings of the 1979 Berlin Colloquium on Global Differential Geometry.  This paper 
contains the original exposition together with some notes by the authors made in 2025 (as indicated in the text)
that give references to descriptions of progress made in the field since the time of the original version of the paper.
The main result of this paper is that every compact isoparametric hypersurface 
$M^n \subset S^{n+1} \subset {\bf R}^{n+2}$ is tight, i.e., every non-degenerate linear height function
$\ell_p$, $p \in S^{n+1}$, has the minimum number of critical points on $M^n$ required by the Morse inequalities.
Since $M^n$ lies in the sphere $S^{n+1}$, this implies that $M^n$ is also taut in $S^{n+1}$, i.e., every non-degenerate 
spherical distance function has the minimum number of critical points on $M^n$.  A  second result is that the focal submanifolds of isoparametric hypersurfaces in $S^{n+1}$ must also be taut.
The proofs of these results are based on
M\"{u}nzner's \cite{Mu}--\cite{Mu2} fundamental work on the structure of a family of isoparametric hypersurfaces 
in a sphere.
\end{abstract}

In his lecture at the 1979 Chern Symposium at Berkeley, N. H. Kuiper pointed out that in all known examples of tight embeddings $M^n \subset S^{n+1}$ (Euclidean sphere), the manifold $M^n$ was diffeomorphic to a sphere or a product of spheres.  He asked if there were more examples, and if so, whether they must necessarily be algebraic.  The main result of this paper is to show that every compact isoparametric hypersurface is tight.  This provides an infinite collection of new examples of tight spherical hypersurfaces as can be seen from the work of Takagi and Takahashi \cite{TT}--\cite{Takagi},
Ozeki and Takeuchi \cite{OT}--\cite{OT2} and more recently the results of Ferus, Karcher and M\"{u}nzner \cite{FKM}.

Although the compact isoparametric hypersurfaces have not yet been classified, M\"{u}nzner \cite{Mu}--\cite{Mu2} has obtained a sufficiently explicit description of their topology and geometry to allow us to prove that they must be tight. 
M\"{u}nzner's proof also establishes the fact that each isoparametric hypersurface is algebraic.  Thus the second part of Kuiper's question remains open.\\

\noindent
{\bf Note 1:} (added by the authors in 2025) {\it Through the work of many researchers, the classification of isoparametric hypersurfaces in spheres has now been completed.  See, for example, the survey article of Chi \cite{Chi-survey} or the book \cite[pp. 144--182]{CR8} for descriptions of the classification.}\\

Our second result is that the focal submanifolds of isoparametric hypersurfaces must also be tight.  This path of exploration leads to some interesting questions concerning tightness and the structure of the focal set which we list at the end of the paper.

Our results also provide many new examples of taut submanifolds (both compact and non-compact) of Euclidean space.  These are obtained as the images of isoparametric hypersurfaces and their focal sets under stereographic projection.\\

\noindent
{\bf 1. Preliminaries.}\\

\noindent
{\bf 1a. Tight and taut immersions.}\\

Let $f:M \rightarrow E^m$ be an immersion of a smooth manifold into Euclidean space.  Let $S^{m-1}$ denote the unit sphere in $E^m$.  For $\xi \in S^{m-1}$, the linear height function $\ell_{\xi}:M \rightarrow {\bf R}$ is defined by 
$\ell_{\xi} (x) = \langle f(x), \xi \rangle$, where $\langle\ ,\ \rangle$ is the Euclidean inner product.  For $p \in E^m$,
the Euclidean distance function $L_p:M \rightarrow {\bf R}$ is defined by $L_p (x) = | f(x) - p |^2$.
One shows using Sard's theorem that almost all of the functions of each type are non-degenerate on $M$.
The immersion $f$ is called {\em tight}, respectively, {\em taut}, if every non-degenerate $\ell_{\xi}$, respectively
$L_p$, has the minimum number of critical points required by the Morse inequalities.  Tightness was first discussed in the form of minimal total absolute curvature by Chern and Lashof \cite{CL1}.  The formulation in terms of critical point theory is due to Kuiper \cite{Ku1}.

For a two-dimensional manifold $M$, tightness, respectively, tautness, is equivalent to the two-piece property of Banchoff with respect to Euclidean hyperplanes \cite{Ban2}, respectively, hyperspheres \cite{Ban1}.  Since a hyperplane can be approximated arbitrarily closely by a hypersphere of large radius, one sees that tautness implies tightness for two-dimensional manifolds.  By a similar argument, one shows that this is true for arbitrary manifolds \cite{CW1}.
Moreover, a taut immersion must be an embedding \cite{CW1}, whereas a tight immersion need not be.

An immersion $f:M\rightarrow E^m$ is called {\em spherical} if the image of $f$ lies in a Euclidean hypersphere
$S^{m-1} \subseteq E^m$.  If $f$ is tight and spherical, it is easy to show that $f$ is taut \cite{Ban1}.  Moreover, using the fact that tautness is a conformal property \cite{CW1}, one then can obtain a taut embedding of $M$ into $E^{m-1}$ by composing $f$ with stereographic projection into $E^{m-1}$.  Thus, {\em tautness is equivalent to the combination of tight and spherical} in the sense that if $M$ is a compact manifold which is not a sphere, then there exists a substantial 
non-spherical taut embedding of $M$ into $E^{m-1}$ if and only if there exists a substantial tight spherical embedding of $M$ into $E^m$.\\

\noindent
{\bf 1b. Distance functions and focal points in the sphere.}\\

It is advantageous at times to consider distance functions in the sphere rather than linear height functions.  Let $f:M \rightarrow S^m$ be an immersion.  For $p \in S^m$, $x \in M$, the spherical distance function $d_p$ has the form
\begin{displaymath}
d_p (x) = \cos^{-1} \langle p, f(x) \rangle.
\end{displaymath}
Clearly $d_p$ and the linear height function $\ell_p$ have the same critical points.  As in the case of Euclidean distance functions \cite[p. 32--38]{Mil}, there is a natural formula for the index of a critical point of $d_p$ in terms of the focal points of $M$ in $S^m$ which we will now develop.

Let $NM$ denote the normal bundle of $M$ in $S^m$ with projection map $\pi$ onto $M$.  The exponential map
$E:NM \rightarrow S^m$ is defined by letting $E(\xi)$ be the point reached by traversing a distance $| \xi |$ along the geodesic in $S^m$ with initial tangent vector $\xi$.
A point $p \in S^m$ is called a {\em focal point of $(M,x)$ of multiplicity} $\nu > 0$ if $p = E(\xi)$, where 
$x = \pi \xi$ and the Jacobian of $E$ has nullity $\nu$ at $\xi$.

Assume now that $|\xi| = 1$ and let $x = \pi \xi$.  A straightforward computation shows that $p = E(t \xi)$ is a focal point of multiplicity $\nu$ of $(M,x)$ if and only if $\cot (t)$ is an eigenvalue of the shape operator $A_\xi$ of multiplicity $\nu$.  Note that each principal curvature $\lambda$ of $A_\xi$ gives rise to two distinct focal points of $(M,x)$ of the form
\begin{displaymath}
p = \cos t \  x + \sin t\ \xi
\end{displaymath}
for $t = \cot^{-1} \lambda$ and $t = \cot^{-1} \lambda - \pi$. (For further discussion, see \cite{CR1}.)

The relationship between focal points and spherical distance functions is embodied in the next result whose proof is completely analogous to the corresponding result in the Euclidean case \cite {Mil}.

\begin{theorem}
\label{theorem-1}
(Index Theorem for spherical distance functions).

\noindent
Let $f:M \rightarrow S^m$ be an immersion and let $p \in S^m$.\\

\noindent
(i) $d_p$ has a critical point at $x \in M$ if and only if $p$ lies on a normal 
geodesic to $f(M)$ at $f(x)$.\\

\noindent
(ii) $x$ is a degenerate critical point of $d_p$ if and only if $p$ is a focal point 
of $(M,x)$.\\

\noindent
(iii) If $x$ is a non-degenerate critical point of $d_p$, then the index of $d_p$ at 
$x$ is equal to the number of focal points (counting multiplicities) on the 
shortest geodesic from $p$ to $f(x)$.
\end{theorem}

Consider the special case when $M \subset S^m$ is a compact embedded hypersurface.  Let $\xi$ be a global field of 
unit normals on $M$.  Then $NM$ is diffeomorphic to ${\bf R} \times M$ and the exponential map has the form
\begin{displaymath}
E(t,x) = \cos t \  x + \sin t\ \xi (x).
\end{displaymath}
From the Index Theorem one immediately obtains the following result.

\begin{corollary}
\label{corollary-1}
Let $M \subset S^m$ be a compact embedded hypersurface.  If $d_p$ is a non-degenerate function on $M$, 
then the number of critical points of $d_p$ on $M$ is equal to the number of inverse images under $E$ of $p$ in the set 
$(-\pi, \pi] \times M$.\\
\end{corollary}

\noindent
{\bf 2. Isoparametric families of hypersurfaces in spheres.}\\

An embedded hypersurface $M^n \subset S^{n+1}$ is called {\em isoparametric} if its principal curvatures are constant.  For the following local considerations, we assume that there is a field of unit normals $\xi$ to $M^n$. A straight-forward
calculation (see, for example, Nomizu \cite{Nom4}) shows that the parallel sets
\begin{displaymath}
M_t = \{  \cos t \  x + \sin t\ \xi (x) \mid x \in M^n \}
\end{displaymath}
are also hypersurfaces for sufficiently small $t$, and moreover each $M_t$ also has constant principal curvatures.  The collection $M_t$ is called an {\em isoparametric family of hypersurfaces}.  In addition, the computation of the principal curvatures of $M_t$ shows that all the hypersurfaces in the family have the same focal set in $S^{n+1}$.

We now briefly cite two examples.  First when $M^n$ is a great or small sphere in $S^{n+1}$, then the parallel
hypersurfaces $M_t$ are also great or small spheres.  Of course, these are the only hypersurfaces with one constant 
principal curvature of multiplicity $n$.  The only examples with two distinct constant principal curvatures are families of products
\begin{displaymath}
M_t = S^k (t) \times S^{n-k} ((1 - t^2)^{1/2}) \subset S^{n+1},\  0 < t < 1.
\end{displaymath}
We refer the reader to Takagi and Takahashi \cite{TT} for a list of the homogeneous examples, and to Ozeki and Takeuchi
\cite{OT}--\cite{OT2} for a large collection of examples including many non-homogeneous ones.

Isoparametric hypersurfaces were first studied extensively by E. Cartan \cite{Car2}--\cite{Car5}.  In his original treatment, Cartan actually began with a different definition which turns out to be equivalent to the above (see, for example,
Nomizu \cite{Nom4} for a concise treatment of the equivalence of the definitions).  Let $\phi:S^{n+1} \rightarrow {\bf R}$
be a differentiable $(C^\infty)$ function.  Then the level set
\begin{displaymath}
M_s = \{x \in S^{n+1} \mid \phi (x) = s \}
\end{displaymath}
is a non-singular hypersurface if grad $\phi \neq 0$ on $M_s$.  Cartan pointed out that if the Laplacian $\Delta \phi$
and the norm of the gradient $|{\rm grad}\ \phi |$ were functions of $\phi$ itself,
then each hypersurface $M_s$ has constant principal curvatures, and hence $\{M_s\}$ is an isoparametric family
although, in general, $s$ is not an arc-length parameter, as $t$ is in the case of parallel hypersurfaces.

Cartan showed that if each $M_s$ has 3 distinct principal curvatures, then the principal curvatures have the same multiplicity, and moreover, each $M_s$ is a level set of the restriction to $S^{n+1}$ of a harmonic homogeneous polynomial of degree 3 on $E^{n+2}$.

In a far-reaching generalization of Cartan's methods and results, M\"{u}nzner \cite{Mu}--\cite{Mu2} showed that each isoparametric hypersurface is algebraic in a precise sense.  We briefly outline his results in the next two
paragraphs.

M\"{u}nzner showed that the number $g$ of distinct constant principal curvatures of an isoparametric hypersurface $M^n$ must be $1,2,3,4$ or 6.  The $2g$ focal points along each normal great circle to $M^n$ are equally spaced at intervals of length $\pi/g$.  There are at most two distinct multiplicities $m_1$ and $m_2$ of distinct principal curvatures.
If $m_1 \neq m_2$, then $g$ is even, and if $\lambda_1 >  \ldots > \lambda_g$ are the distinct principal curvatures, then 
$\lambda_i$ has multiplicity $m_1$ if $i$ is odd and $m_2$ if $i$ is even.  Finally, there is a homogeneous polynomial $F:E^{n+2} \rightarrow {\bf R}$ of degree $g$ satisfying
\begin{eqnarray}
|{\rm grad}\ F |^2 & = & g^2 r^{2g-2} \quad ({\rm gradient \ in}\ E^{n+2}) \nonumber \\
\Delta F & = &  c r^{g-2}  \quad\  \ ({\rm Laplacian \ in}\ E^{n+2}) \nonumber
\end{eqnarray}
such that $M^n$ is an open subset of some level set $M_s$ of the restriction of $F$ to $S^{n+1}$. Here, $r = |x|$ and 
\begin{displaymath}
c = \left(\frac{m_1 - m_2}{2} \right) g^2.
\end{displaymath}
Thus, if $m_1 = m_2$, then $F$ is harmonic, as Cartan knew.  Let $V$ be the restriction of $F$ to $S^{n+1}$.  
M\"{u}nzner's construction shows that $V$ maps $S^{n+1}$ into $[-1,1]$.  For $-1 < s < 1$, $M_s$ is a
compact isoparametric hypersurface, while $M_1$ and $M_{-1}$ are focal submanifolds of dimensions $n-m_1$ 
and $n-m_2$, respectively.\\

\noindent
{\bf Note 2:} (added by the authors in 2025) {\it M\"{u}nzner also proved that all of the level sets of $F$, including the two
focal submanifolds, are connected.}\\

A crucial part of M\"{u}nzner's construction from the point of view of this paper is the behavior of $V$ with respect
to the normal exponential map.  M\"{u}nzner makes a change of parameter for the normal exponential map which
simplifies his exposition.  Suppose $\cot \theta$ is the largest positive principal curvature of $M$, and define
$E: {\bf R} \times M \rightarrow S^{n+1}$ by
\begin{displaymath}
E(t,x) = \cos \  (\theta - t)\  x + \sin \  (\theta - t)\  \xi (x).
\end{displaymath}
As a result of M\"{u}nzner's construction, 
\begin{equation}
\label{eq:1}
V(E(t,x)) = \cos (gt).
\end{equation}
Thus $E(0,x)$ and $E(\pi/g,x)$ are focal points of $(M^n,x)$ on $M_1$ and $M_{-1}$, respectively.  As 
M\"{u}nzner points out,
\begin{displaymath}
E: \left( 0, \frac{\pi}{g} \right) \times M \rightarrow S^{n+1} - \{M_1 \cup M_{-1} \}
\end{displaymath}
is a diffeomorphism.  By \eqref{eq:1}, this is also true if $( 0, \pi/g)$ is replaced by an interval 
$( k \pi/g, (k+1)\pi/g)$.  Consequently, if $d_p$ is non-degenerate, i.e., $p \notin M_1 \cup M_{-1}$,
then $p$ has $2g$ distinct inverse images under $E$ in $(- \pi, \pi] \times M^n$.  Combining this with the
Corollary in Section 1, we have\\

\noindent
{\bf Proposition 2:}  {\it Let $M^n \subset S^{n+1}$ be a compact isoparametric hypersurface.  Then every non-degenerate distance function $d_p$ has $2g$ critical points on $M^n$.}\\

\noindent
Note also that the indices of the critical points can be determined from the multiplicities $m_1$ and $m_2$.

Using the fact that $M^n$ divides $S^{n+1}$ into two disk bundles, an $(m_1 + 1)$-dimensional disk bundle over $M_1$ and an $(m_2 + 1)$-dimensional disk bundle over $M_{-1}$,  M\"{u}nzner computes the cohomology of $M^n$, $M_1$ and $M_{-1}$.  The sum of the ${\bf Z}_2$-Betti numbers of each isoparametric hypersurface is $2g$, while the sum 
of the ${\bf Z}_2$-Betti numbers of $M_1$ or $M_{-1}$ is $g$.  We now state the main result.

\begin{theorem}
\label{theorem-2}
(a) Every isoparametric hypersurface in $S^{n+1}$ is tight.\\
(b) Each of the focal submanifolds of an isoparametric hypersurface in $S^{n+1}$ is tight.
\end{theorem}

\begin{proof}
The proof of (a) is immediate from Proposition 2 since each $d_p$, and hence each $\ell_p$, has the minimum number of critical points required by the ${\bf Z}_2$-homology of $M$.

(b) First note that the set of $p$ such that $d_p$ is non-degenerate on $M_1$ is again $S^{n+1} - \{M_1 \cup M_{-1}\}$.
This follows from M\"{u}nzner's computation of the shape operators for $M_1$.  A great circle through $p$ which intersects $M_1$ orthogonally is also normal to the isoparametric family of hypersurfaces $M_s, \ -1 < s < 1$.  However,
by M\"{u}nzner's construction, there is only one great circle through $p$ which is normal to the $M_s$.  This great circle intersects each hypersurface $M_s$  $2g$ times. Thus, all of the critical points of $d_p$ on $M_1$ lie on one great circle.
There are $g$ points of $M_1$ on each normal great circle, so each non-degenerate $d_p$ has $g$ critical points on $M_1$.  Comparing this with the sum of the ${\bf Z}_2$-Betti numbers of $M_1$ proves that $M_1$ is tight.  To complete the proof of (b), we note that the argument remains equally valid when $M_1$ and $M_{-1}$ are interchanged.
\end{proof}

\noindent
{\bf Remark:}
 Nomizu \cite{Nom4} proved that the focal submanifolds must also be minimal in the sense of zero mean curvature.\\

\noindent
{\bf 3.  Examples and classifications of tight hypersurfaces in spheres.}\\

As mentioned earlier, spheres and products of two spheres are examples of isoparametric hypersurfaces in $S^{n+1}$
which were previously known to be tight.  Nomizu and Rodriguez have shown \cite{NR} that the only
tight embeddings of $S^n$ into $S^{n+1}$ are the round spheres.

For $M^n = S^k (s) \times S^{n-k} ((1 - s^2)^{1/2})$ in $S^{n+1}$, note that the focal set consists of two totally geodesic 
spheres 
\begin{displaymath}
M_1 = S^k (1) \times \{0\}, \quad M_{-1} = \{ 0 \} \times S^{n-k} (1).
\end{displaymath}  
The images of $M^n$ under a conformal transformation of $S^{n+1}$ (or under stereographic projection into $E^{n+1}$) are called {\em cyclides of Dupin}.
The cyclides were characterized in \cite{CR2} as the only codimension one tight spherical embeddings of a manifold $M^n$ with the same integral homology as $S^k \times S^{n-k}$.

As we have stated, there are infinitely many examples of isoparametric hypersurfaces.  Here we briefly mention two homogeneous examples which are noteworthy because of their familiarity.  Cartan showed that every isoparametric 
$M^3 \subset S^4$ must be of the form of the first example.\\

\noindent
{\bf Example 1:}\\

\noindent
Let $S^4$ be the set of $3 \times 3$ symmetric matrices of norm 1 and trace 0.  Let $SO (3)$ act on $S^4$ by conjugation.  
For $0 < t < \pi/3$, let $M_t^3$ be the orbit of the diagonal matrix whose entries are
\begin{displaymath}
\sqrt{\frac{2}{3}} \ \left( \cos (t - \frac{\pi}{3}),\  \cos (t + \frac{\pi}{3}), \ \cos (t + \pi) \right).
\end{displaymath}  
Each $M_t^3$ is an isoparametric hypersurface with 3 distinct principal curvatures.  $M_t^3$ is diffeomorphic to
$SO(3)/ {\bf Z}_2 \times {\bf Z}_2$, where ${\bf Z}_2 \times {\bf Z}_2$ is the group of diagonal matrices with $\pm 1$ along the diagonal and determinant 1.  Setting $t = 0$ or $t = \pi/3$ yields a Veronese surface.  Specifically,
\begin{displaymath}
M_0 = {\rm orbit\ of\ diagonal\ matrix\ } \frac{1}{\sqrt{6}} \ (1,1, -2)
\end{displaymath}  
and
\begin{displaymath}
M_{\frac{\pi}{3}} = {\rm orbit\ of\ diagonal\ matrix\ } \frac{1}{\sqrt{6}} \ (2, -1, -1)
\end{displaymath}  
are Veronese surfaces whose union is the focal set of each $M_t^3$.\\

\noindent
{\bf Example 2:}\\

\noindent
This is worked out in detail by Nomizu \cite{Nom4}.  We will simply note that $M^{2n} \subset S^{2n+1}$ has four distinct principal curvatures of multiplicities $n-1, 1, n-1, 1$.  (Takagi  \cite{Takagi}  proved that this is the only example with this configuration of multiplicities.) $M^{2n}$ is diffeomorphic to $V_{n+1,2}  \times S^1$, where $V_{n+1,2}$ is the Stiefel
manifold of orthonormal 2-frames in $E^{n+1}$.  One sheet of the focal set is a naturally embedded $V_{n+1,2}$, 
while the other is diffeomorphic to $S^n \times S^1$.\\

\noindent
{\bf Remark:}
Aside from the spheres and cyclides, the only known tight hypersurfaces in spheres known until now were of
the following type constructed by Carter and West \cite[p. 718]{CW1}.  Let $M$ be the standard product,
\begin{displaymath}
M = S^{k_1} \times \ldots \times S^{k_r} \subset S^{k_1 + \ldots + k_r +r -1 =m}.
\end{displaymath}  
Then $M$ is a codimension $r-1$ tight embedding.  A tube $\widetilde{M}$ of sufficiently small radius is a tight codimension one embedding of $M \times S^{r-2}$ into $S^m$.\\

\noindent
{\bf 4.  Tight embeddings and focal sets.}\\

Let $M^n \subset S^{n+1}$ be an isoparametric hypersurface or the image of one under a conformal transformation
of $S^{n+1}$.  Then each principal curvature $\lambda$ of $M^n$ satisfies\\

\noindent
$(*)$ $\lambda$ has constant multiplicity $\nu$ on $M^n$, and $\lambda$ is constant along the leaves of its principal foliation $T_\lambda$.\\

\noindent
{\bf Note 3:} (added by the authors in 2025)  {\it Hypersurfaces $M^n$ in $S^{n+1}$ or $E^{n+1}$ whose principal curvatures all satisfy
property $(*)$ are now called proper Dupin hypersurfaces, using the terminology introduced by Pinkall  \cite{P4}.
The field of  Dupin hypersurfaces has been extensively studied by many researchers (see,
for example \cite[pp. 233--342]{CR8})}.\\

\noindent
Note that the second condition is automatically satisfied if $\nu > 1$ \cite[p. 372--373]{Ryan1}.  In a previous
paper \cite{CR1}, it was shown that $(*)$ implies\\

\noindent
(1) The leaves of $T_\lambda$ are $\nu$-dimensional Euclidean spheres.\\

\noindent
(2) The sheet $f_\lambda (M)$ of the focal set of $M^n$ corresponding to $\lambda$ is an $(n - \nu)$-dimensional submanifold of $S^{n+1}$, and $M^n$ is a $\nu$-sphere bundle over $f_\lambda (M)$.\\

The same theorem holds for $M^n \subset E^{n+1}$ except that the leaves of $T_\lambda$ can be $\nu$-spheres or 
$\nu$-planes and each component $U$ of the subset of $M^n$ on which $\lambda \neq 0$ is a $\nu$-sphere
bundle over $f_\lambda (U)$.  H. Reckziegel has generalized these results in a recent paper \cite{Reck2}.
The compact cyclides and the complete non-compact cyclide in $E^{n+1}$ obtained by taking
the pole of the stereographic projection 
to be on $S^k \times S^{n-k}$ were characterized in terms of the condition $(*)$ as follows in \cite{CR2}
and \cite{CR5}.\\

\noindent
{\bf Theorem}. {\it Let $M^n \subset E^{n+1}$ be a connected, complete hypersurface having two distinct principal curvatures at each point, both of which satisfy $(*)$.  Then $M^n$ is a cyclide.}\\

\noindent
{\bf Questions:}\\

\noindent
1. Can any other isoparametric hypersurfaces and their conformal images be characterized similarly?   In particular, Cartan's $M^3 \subset S^4$.\\

\noindent
2.  For $M^n \subset S^{n+1}$, does the assumption that each principal curvature of $M$ satisfies $(*)$ imply that $M^n$ 
is a conformal image of an isoparametric hypersurface?\\

\noindent
3.  Can any of the isoparametric hypersurfaces be characterized (up to conformality) as the only tight hypersurfaces
having a given homology, as $S^n$ and $S^k \times S^{n-k}$ have been?\\

\noindent
{\bf Note 4:} (added by the authors in 2025)  {\it Regarding Questions 1 and 2, Miyaoka \cite{Mi1} proved that a compact,
connected proper Dupin hypersurface $M^n \subset S^{n+1}$ with $g=3$ principal curvatures is equivalent by a Lie sphere transformation to an isoparametric hypersurface, but it is not necessarily M\"{o}bius equivalent to an isoparametric hypersurface.
Later, Pinkall and Thorbergsson \cite{PT1}, and Miyaoka and Ozawa \cite{MO}, gave separate constructions of compact,
connected proper Dupin hypersurfaces in $S^{n+1}$ that are not 
equivalent by a Lie sphere transformation to an isoparametric 
hypersurface in $S^n$. 
All three of the questions above are related to classifications of compact proper
Dupin hypersurfaces $M^n$ in the sphere $S^{n+1}$.  See  the book \cite[pp. 308--322]{CR8} for a survey of results
on this topic.}\\

In a recent paper \cite{CW2}, Carter and West defined an embedded submanifold $M \subset E^m$ to be {\em totally focal}
if every Euclidean distance function has the property that either all of its critical points are non-degenerate or all
of its critical points are degenerate.  If a totally focal submanifold is defined analogously, the following is clear 
from M\"{u}nzner's construction:\\

\noindent
{\bf Proposition 3:}  
{\it Every isoparametric hypersurface $M^n \subset S^{n+1}$ is totally focal.}\\

\noindent
Carter and West proved that the only totally focal hypersurfaces in $E^{n+1}$ are spheres, planes, and spherical cylinders
$S^k \times E^{n-k}$.  These are precisely the hypersurfaces of $E^{n+1}$ which are isoparametric, i.e., have constant principal curvatures.  Thus we ask\\

\noindent
4.  Are the isoparametric hypersurfaces in $S^{n+1}$ the only totally focal hypersurfaces of $S^{n+1}$?\\

\noindent
{\bf Note 5:} (added by the authors in 2025)  {\it Question 4 was answered in the affirmative by Carter and West \cite{CW4}.  Carter and West also discussed totally focal submanifolds of higher codimension in a
later paper \cite{CW8}.
For a survey of results on the relationship between isoparametric and totally
focal submanifolds,  see  \cite[pp. 139---140]{CR8}.}\\

\noindent Thomas E. Cecil

\noindent Department of Mathematics and Computer Science

\noindent College of the Holy Cross

\noindent Worcester, MA 01610

\noindent email: tcecil@holycross.edu\\

\noindent Patrick J. Ryan

\noindent Department of Mathematics and Statistics

\noindent McMaster University

\noindent Hamilton, Ontario, Canada L8S4K1

\noindent email: ryanpj@mcmaster.ca

\end{document}